\documentclass[12pt, reqno]{amsart}
\usepackage{amsmath, amsthm, amscd, amsfonts, amssymb, graphicx, color, enumerate, xcolor,  mathrsfs, latexsym}
\usepackage[bookmarksnumbered, colorlinks, plainpages]{hyperref}
\hypersetup{colorlinks=true,linkcolor=red, anchorcolor=green, citecolor=cyan, urlcolor=red, filecolor=magenta, pdftoolbar=true}

  \definecolor{cupgreen}{rgb}{0,0.498,0.208}
  \definecolor{cupblue}{rgb}{0,0,.5}
  \definecolor{cupred}{rgb}{1,0.04,0}
  \definecolor{cuppink}{rgb}{0.925,0,0.545}
  \definecolor{cupmagenta}{rgb}{0.624,0.161,0.424}
  \definecolor{cupbrown}{rgb}{0.71,0.212,0.133}

  \definecolor{cupgreen}{rgb}{0,0,0}
  \definecolor{cupblue}{rgb}{0,0,0}
  \definecolor{cupred}{rgb}{0,0,0}
  \definecolor{cuppink}{rgb}{0,0,0}
  \definecolor{cupmagenta}{rgb}{0,0,0}
  \definecolor{cupbrown}{rgb}{0,0,0}

\definecolor{TITLE}{rgb}{0,0,0}
\definecolor{AUTHOR1}{rgb}{0.00,0.59,0.00}
\definecolor{AUTHOR2}{rgb}{0.50,0.00,1.00}
\definecolor{SECTION}{rgb}{0.50,0.00,1.00}
\definecolor{FOOTTITLE}{rgb}{0.00,0.50,0.75}
\definecolor{THM}{rgb}{0.8,0,0.1}
\definecolor{SEC}{rgb}{0,0,1}

\makeatletter
\@namedef{subjclassname@2010}{%
  \textup{2010} Mathematics Subject Classification}
\makeatother
\normalsize
\newtheorem{theorem}{{\color{THM} Theorem}}[section]
\newtheorem{lemma}[theorem]{{\color{THM}Lemma}}
\newtheorem{conjecture}[theorem]{{\color{THM}Conjecture}}
\newtheorem{proposition}[theorem]{{\color{THM}Proposition}}
\newtheorem{corollary}[theorem]{{\color{THM}Corollary}}
\theoremstyle{definition}

\newtheorem{example}[theorem]{{\color{THM}Example}}

\newtheorem{remark}[theorem]{{\color{THM}Remark}}

\newcommand{\A}{\mathcal A}
\newcommand{\B}{\mathcal B}
\newcommand{\D}{\mathcal D}
\newcommand{\G}{\mathcal G}
\newcommand{\M}{\mathcal M}
\newcommand{\V}{\mathcal V}
\newcommand{\K}{\mathcal L}
\newcommand{\W}{\mathcal W}

\newcommand{\N}{\mathcal N}

\newcommand{\bea}{\begin{eqnarray*}}
\newcommand{\eea}{\end{eqnarray*}}
\usepackage{amscd}
\begin{document}
\noindent \textcolor[rgb]{0.99,0.00,0.00}{}\\[.5in]
\title[More on Lie derivations  of   generalized matrix algebras] {\color{TITLE} More on Lie derivations  of   generalized matrix algebras}
\author[Mokhtari, Ebrahimi Vishki]{A.H. Mokhtari  and H.R. Ebrahimi Vishki}
\address{Department of Pure Mathematics, Ferdowsi University of Mashhad, P.O. Box 1159, Mashhad 91775, Iran}
\email{{amirmkh2002@yahoo.com}}
\address{ Department of Pure Mathematics and  Center of Excellence in Analysis on Algebraic Structures (CEAAS), Ferdowsi University of Mashhad, P.O. Box 1159, Mashhad 91775, IRAN.}
\email{{vishki@um.ac.ir}}
\begin{abstract} Motivated by  the Cheung's elaborate work [Linear Multilinear Algebra,
\textbf{51}  (2003), 299-310], we investigate the construction of a Lie derivation on a generalized matrix algebra and apply it to give a characterization for such a Lie derivation to be proper. Our approach not only provides a direct proof for some known  results in the theory, but also it presents several sufficient conditions assuring the properness of Lie derivations on certain generalized matrix algebras.
  \end{abstract}
\maketitle
\section{\color{SEC}Introduction and Preliminaries}
Let $\A$ be a unital algebra (over a commutative unital ring ${\bf R}$) and $\M$ be an $\A-$module. A linear  mapping $\D:\A\rightarrow\M$  is said to be a derivation if
\[\D(ab)=\D(a)b+a\D(b)\qquad (a, b\in \A).\]
\noindent A linear  mapping $\K:\A\rightarrow\M$ is called a Lie derivation if
\[\K([a,b])=[\K(a),b]+[a,\K(b)]\qquad (a, b\in \A),\]
where $[a, b]=ab-ba$.  Every  derivation is trivially a Lie derivation. If $\D:\A\rightarrow\A$ is a derivation and
$\tau:\A\rightarrow Z(\A)$ is a linear map, where $Z(\A)$ denotes the center of $\A$, then $\D+\tau$ is a Lie derivation if and only if $\tau$ vanishes at commutators (i.e.  $\tau([a,b])=0,$ for all $a,b\in\A$).  Lie derivations of this form are called proper. Clearly   a Lie derivation $\K$ is  proper if and only if  $\K=\D+\tau$ for some derivation $\D$ and  a linear center valued  map $\tau$ on $\A$  (i.e. $\tau(\A)\subseteq Z(\A)$).
The fundamental question  is  that under what conditions a Lie  derivation on an algebra is proper.
We say that an algebra $\A$ has ``{\it Lie Derivation Property}'' if every Lie derivation from  $\A$ into itself is proper.

In this paper we study the Lie derivation property  for the general matrix algebras. First we briefly introduce a general matrix algebra. A Morita  context  $(\A, \B, \M, \N, \Phi_{\M\N}, \Psi_{\N\M})$ consists of two unital algebras $\A$, $\B$, an $(\A,\B)-$module $\M$, a $(\B,\A)-$module $\N$, and two module homomorphisms  $\Phi_{\M\N}:\M\otimes_\B \N\longrightarrow \A$ and $\Psi_{\N\M}:\N\otimes_\A \M\longrightarrow \B$ satisfying the following commutative diagrams:
\begin{equation*}\label{Dia}\begin{CD}
\M\otimes_\B\N\otimes_\A\M @ >\Phi_{\M\N}\otimes I_\M >> \A\otimes_\A \M\\
@ VV I_\M\otimes\Psi_{\N\M} V @ VV\cong V\\
\M\otimes_\B\B @ >\cong >> \M
\end{CD}
\end{equation*}
and
\begin{equation*}\begin{CD}
\N\otimes_\A\M\otimes_\B\N @ > \Psi_{\N\M}\otimes I_\N >> \B\otimes_\B\N\\
@ VV I_\N\otimes\Phi_{\M\N} V @ VV\cong V\\
\N\otimes_\A\A @ > \cong >>\N.
\end{CD}
\end{equation*}
For a Morita context  $(\A, \B, \M, \N, \Phi_{\M\N},\Psi_{\N\M})$, the set
\[\mathcal{G}=\left(
\begin{array}{cc}
\A & \M \\
\N & \B \\
\end{array}
\right)=\Bigg\{\left(
\begin{array}{cc}
a & m \\
n & b \\
\end{array}
\right)\Big| a\in \A, m\in \M, n\in \N, b\in \B\Bigg\}\]
forms an algebra under the usual matrix operations, where at least one of two modules $\M$
and $\N$ is distinct from zero. The algebra $\G$  is called a generalized matrix algebra. In the case where  $\N=0$, $\mathcal{G}$ becomes the so-called triangular algebra ${\rm Tri}(\A,\M,\B),$ whose Lie derivation property is  investigated by Cheung \cite{CH1}. Generalized matrix algebras were  first introduced by Sands  \cite{S}, where he studied various radicals of algebras occurring in Morita contexts.

A direct verification reveals that the center $Z(\mathcal{G})$ of $\mathcal{G}$ is
\[Z(\mathcal{G})=\{a\oplus b| a\in Z(\A), b\in Z(\B),\ am=mb, na=bn \quad{\rm for\ all} \quad m\in \M, n\in \N\},\]
where $a\oplus b=\left(
\begin{array}{cc}
a & 0 \\
0 & b \\
\end{array}
\right)\in \mathcal{G}.$ We also define two natural projections $\pi_\A:\mathcal{G}\longrightarrow \A$ and $\pi_\B:\mathcal{G}\longrightarrow \B$ by
\[\pi_\A:\left(
         \begin{array}{cc}
           a & m \\
           n & b \\
         \end{array}
       \right)\mapsto a
 \quad {\rm and}\quad  \pi_\B:\left(
         \begin{array}{cc}
           a & m \\
           n & b \\
         \end{array}
       \right)\mapsto b.\]\\
It is clear that, $\pi_\A(Z(\mathcal{G}))\subseteq Z(\A)$ and $\pi_\B(Z(\mathcal{G}))\subseteq Z(\B)$. Further, similar to \cite[Proposition 3]{CH1} one can show that,  if  $\M$ is a faithful $(\A,\B)-$module (i.e. $a\M=0$  necessities $a=0$ and $\M b=0$  necessities $b=0$), then there exists a unique algebra isomorphism
 \[\varphi:\pi_\A(Z(\mathcal{G}))\longrightarrow \pi_\B(Z(\mathcal{G}))\] such that $am=m\varphi(a)$ and $\varphi(a)n=na$ for all $m\in \M$, $n\in \N$; or equivalently, $a\oplus\varphi(a)\in Z(\G)$ for all $a\in\A.$

 Martindale \cite{Ma} was the first one who investigated the Lie derivation property of   certain primitive rings.
Cheung \cite{CH2} initiated the study of various mappings on triangular algebras; in particular, he  studied
the Lie derivation property for  triangular algebras in  \cite{CH1}. Following  his work \cite{CH1}, Lie derivations on a wide variety of algebras have been studied by many authors (see \cite{B, Br, DW, LW, LL, LJ, Mk, ME, W, WW} and the references therein). The main result of Cheung \cite{CH1} has recently extended by Du and Wang \cite{DW} for a generalized matrix algebra.  They showed that \cite[Theorem 1]{DW} a general matrix algebra $\mathcal{G}$
with $\M$ faithful enjoys the Lie derivation property  if  $\pi_{\A}(Z(\mathcal{G}))=Z(\A), \pi_{\B}(Z(\mathcal{G}))=Z(\B)$, and
   either $\A$ or $\B$ does not contain nonzero central ideals.
This result  developed for Lie $n-$derivation in \cite[Theorem 1]{WW}, (see also \cite[Theorem 2.1]{W}).

In this paper we use the construction  of Lie derivations of a generalized matrix algebra $\G$ (Proposition \ref{F})  to give a criterion for properness  Lie derivations on $\G$ (see Theorem \ref{HH} and Corollary \ref{H}). We then deduce not only, as a byproduct, the aforementioned result of Du and Wang (Theorem \ref{ideal}) but also we provide some alternative sufficient conditions ensuring the Lie derivation property for $\G$ (Theorems \ref{domain}, \ref{strong}). We then state our main result, Theorem \ref{main},  giving some sufficient conditions assuring the Lie derivation property for a generalized matrix algebra. In the last section, we include some applications of our results to some special generalized matrix algebras such as triangular algebras,  the full matrix algebras and the algebras of operators on a Banach space.
\section{\color{SEC} Proper Lie derivations}
From now on, we  assume that the modules $\M$ and $\N$ appeared in the generalized matrix algebra $\G$ are $2-$torion free. Recall that  a module $\M$ is called $2-$torsion free if $2m=0$ implies $m=0$ for any $m\in\M$.

We commence with the following    result providing  the construction of  (Lie) derivations of a generalized matrix algebra. It needs a standard argument, however Li and Wei {\cite{LW}} gave a complete proof for the first  two parts of it.
\begin{proposition}[]\label{F}
Let $\G=\left(\begin{array}{cc}
\mathcal{A} & \M\\
\N & \B\\
\end{array}\right)$  be a  generalized matrix algebra. Then

$\bullet$ {\rm({\cite[Proposition  4.1]{LW}}).} A linear mapping $\K$ on $\G$ is a Lie derivation if and only if it has the presentation
\begin{eqnarray}\label{gli}
\K\left(\begin{array}{cc}
a & m\\
n & b\\
\end{array}\right)=\left(\begin{array}{cc}
P(a)-mn_0-m_0n+h_\B(b) & am_0-m_0b+f(m)\\
n_0a-bn_0+g(n)& h_\A(a)+n_0m+nm_0+Q(b)\\
\end{array}\right),
\end{eqnarray}
for some $m_0 \in \mathcal{M}, n_0\in\mathcal{N}$ and some linear maps
$P:\mathcal{A}\longrightarrow\mathcal{A}, Q:\B\longrightarrow\B, f:\mathcal{M}\longrightarrow\mathcal{M},  g:\mathcal{N}\longrightarrow\mathcal{N},  h_\B:\B\longrightarrow Z(\mathcal{A})$ and
$h_\A:\mathcal{A}\longrightarrow Z(\B)$
satisfying the following conditions:
\begin{enumerate}[\hspace{1em}\rm (a)]
\item $P, Q$ are Lie derivations;
\item $h_\B([b,b'])=0,\ h_\A([a,a'])=0$;
\item $f(am)=P(a)m-mh_\A(a)+af(m),\ f(mb)=mQ(b)-h_\B(b)m+f(m)b$;
\item $g(na)=nP(a)-h_\A(a)n+g(n)a,\ g(bn)=Q(b)n-nh_\B(b)+bg(n)$;
\item $P(mn)-h_\B(nm)=mg(n)+f(m)n,\ Q(nm)-h_\A(mn)=g(n)m+nf(m)$;
\end{enumerate}
for all $a,a'\in \mathcal{A}, b,b'\in\B, m\in\M$ and $n\in\mathcal{N}$.

$\bullet$ {\rm ({\cite[Proposition 4.2]{LW}}).} A linear mapping $\D$ on $\G$ is a derivation if and only if it has the presentation
\begin{eqnarray}\label{gd}
\D\left(\begin{array}{cc}
a & m\\
n & b\\
\end{array}\right)=\left(\begin{array}{cc}
P'(a)-mn_0-m_0n & am_0-m_0b+f'(m)\\
n_0a-bn_0+g'(n) & n_0m+nm_0+Q'(b)\\
\end{array}\right),
\end{eqnarray}
where $m_0 \in \M,n_0\in \mathcal{N}, P':\mathcal{A}\longrightarrow \mathcal{A}, Q':\B\longrightarrow \B,
f':\M \longrightarrow\M$ and $g':\mathcal{N}\longrightarrow\mathcal{N}$
are linear maps satisfying the following conditions:
\begin{enumerate}[\hspace{1em}$\rm (a')$]
\item $P', Q'$ are  derivations;
\item $f'(am)=P'(a)m+af'(m),\ f'(mb)=mQ'(b)+f'(m)b$;
\item $g'(na)=nP'(a)+g'(n)a,\ g'(bn)=Q'(b)n +bg'(n)$;
\item $P'(mn)=mg'(n)+f'(m)n,\ Q'(nm)=g'(n)m+ nf'(m)$;
\end{enumerate}
 for all $a\in\mathcal{A}, b\in\B, m\in\M$ and $n\in\mathcal{N}$.

 $\bullet$ A linear mapping $\tau$ on $\G$  is center valued  and vanishes at commutators (i.e. $\tau(\mathcal{G})\subseteq Z(\mathcal{G})$ and $\tau([\mathcal{G},\mathcal{G}])=0$) if and only if $\tau$ has the  presentation
 \begin{eqnarray}\label{cv}
\tau \left(
  \begin{array}{cc}
    a & m \\
    n & b \\
  \end{array}
\right)
=
\left(
  \begin{array}{cc}
    \ell_\A(a)+h_\B(b) & 0 \\
    0                         & h_\A(a)+\ell_\B(b) \\
  \end{array}
\right),
\end{eqnarray}
where $\ell_\A:A\longrightarrow Z(A), h_\B:B\longrightarrow Z(A), h_\A:A\longrightarrow Z(B), \ell_\B:B\longrightarrow Z(B)$ are linear maps vanishing at commutators, having the  following properties:
\begin{enumerate}[\hspace{1em}\rm ($a''$)]
\item  $\ell_\A(a)\oplus h_\A(a)\in Z(\G)$ and  $h_\B(b)\oplus \ell_\B(b)\in Z(\G),$ for all $a\in\A, b\in\B;$
\item $ \ell_\A(mn)=h_\B(nm)$ and $h_\A(mn)=\ell_\B(nm),$ for all $m\in\M, n\in\N.$
\end{enumerate}
\end{proposition}
Following the method of  {\cite[Theorem 6]{CH1}}, in the next theorem we give a necessary and sufficient condition for a Lie derivation on a generalized matrix algebra $\G$ to be proper.
\begin{theorem}\label{HH}
Let $\mathcal{G}$ be a generalized matrix algebra.
A Lie derivation $\K$ on $\mathcal{G}$   of the form presented in \eqref{gli}
is proper if and only if there exist linear  mappings
$\ell_\A:\A \longrightarrow Z(\A) $ and
 $\ell_\B:\B \longrightarrow Z(\B) $
  satisfying the following conditions:
 \begin{enumerate}[\hspace{3em}$\rm (A)$]
    \item $P-\ell_\A$ and $Q-\ell_\B$ are derivations on $\A$ and $\B$, respectively;
    \item   $\ell_\A(a)\oplus h_\A(a)\in Z(\G)$ and  $h_\B(b)\oplus \ell_\B(b)\in Z(\G),$ for all $a\in\A, b\in\B;$
    \item $\ell_\A(mn)=h_\B(nm)$ and $\ell_\B(nm)=h_\A(mn),$  for all $ m \in\M, n\in\N.$
\end{enumerate}
\end{theorem}
\begin{proof} For the sufficiency, employing the characterization of $\K$ as in \eqref{gli},  we define $\D$ and $\tau$ by
$$\D\left(
  \begin{array}{cc}
    a & m \\
    n & b \\
  \end{array}
\right)
=
\left(
  \begin{array}{cc}
    (P-\ell_\A)(a)-mn_0-m_0 n & am_0-m_0 b+f(m) \\
    n_0a-bn_0+g(n) & n_0 m+nm_0+(Q - \ell_\B)(b) \\
  \end{array}
\right)
$$
and
$$\tau \left(
  \begin{array}{cc}
    a & m \\
    n & b \\
  \end{array}
\right)
=
\left(
  \begin{array}{cc}
    \ell_\A(a)+h_\B(b) & 0 \\
    0                         & h_\A(a)+\ell_\B(b) \\
  \end{array}
\right)
.$$
Then, by Proposition \ref{F},  a direct verification reveals  that  $\D$ is a derivation, $\tau$ is center valued  and  $\K=\D+\tau.$

 For the necessity, suppose that $\K$ is of the form $\D+\tau $, where $\D$ is a derivation and $\tau$ maps into
$Z(\mathcal{G}) $. Applying the presentations \eqref{gli}, \eqref{gd} for $\K$ and $\D$, respectively,  we get $\tau=\K-\D$ as
\[\tau\left(\begin{array}{cc}
a & m\\
n & b
\end{array}\right)=\left(\begin{array}{cc}
(P-P')(a)+h_\B(b) & 0 \\
  0 & h_\A(a)+(Q-Q')(b)
\end{array}\right).\]
By setting $\ell_\A=P-P', \ell_\B=Q-Q',$ a direct verification shows that $\ell_\A, \ell_\B$ are our desired maps satisfying the required conditions.
\end{proof}
\begin{remark}\label{faith}  In the case where $\M$ is  faithful as an $(\A,\B)-$module, then  we  have the following  simplifications in Proposition \ref{F} and Theorem \ref{HH}:
\begin{enumerate}[\hspace{1em}\rm (i)]
\item In Proposition \ref{F}, the condition (b)  becomes redundant as it can be followed by (a) and (c). Indeed, for $a, a'\in\A, m\in\M,$ from (c) we get
\begin{equation}\label{2}
f([a,a']m)=P([a,a'])m-mh_\A([a,a'])+[a,a']f(m).
\end{equation}
On the other hand, employing  (c) and then (a), we have
\begin{eqnarray}\label{22}
f([a,a']m)&=&f(aa'm-a'am)\notag\\
&=&P(a)a'm-a'mh_\A(a)+af(a'm)-\big(P(a')am-amh_\A(a')+a'f(am)\big)\notag\\
&=&P(a)a'm-a'mh_\A(a)+a(P(a')m-mh_\A(a')+a'f(m))\notag\\
&&-\big(P(a')am-amh_\A(a')+a'(P(a)m-mh_\A(a)+af(m))\big)\notag\\
&=&[P(a), a']m+[a,P(a')]m+[a,a']f(m)\notag\\
&=&P([a,a'])m+[a,a']f(m).
\end{eqnarray}
Comparing the equations \eqref{2} and \eqref{22} reveals that $mh_\A([a,a'])=0,$ for every $m\in\M,$ and the  faithfulness of $\M$ (as a right $\B-$module) implies that $h_\A([a,a'])=0,$ as claimed.
\item In Proposition \ref{F}, the condition $\rm (a')$ can  be dropped as it can be derived from $\rm (b')$ by a similar argument as in (i), (see {\cite[Page 303]{CH1}}).
\item  The same reason  as in (ii) shows that    the condition $\rm (A)$ in Theorem \ref{HH}, stating that $P-\ell_\A$ and $Q-\ell_\B$ are derivations,   is  superfluous.
\end{enumerate}
\end{remark}
 As a consequence of Theorem \ref{HH} we have the following criterion characterizing  Lie derivation property for a generalized matrix algebra $\G.$
\begin{corollary}\label{H}
Let $\mathcal{G}$ be a generalized matrix algebra and let
 $\K$ be a Lie derivation on $\mathcal{G}$  of the form presented in \eqref{gli}. If $\K$ is proper then
\begin{enumerate}[\hspace{1em}$\rm (A')$]
  \item $h_\A(\A) \subseteq \pi_\B(Z(\mathcal{G}))$, $h_\B(\B)\subseteq \pi_\A(Z(\mathcal{G})),$ and
  \item $h_\B(nm)\oplus h_\A(mn) \in Z(\mathcal{G}),$ for all $m\in \M, n\in \N.$
\end{enumerate}
The converse   holds true  in the case when $\M$ is faithful.
\end{corollary}
\begin{proof}
The necessity follows trivially from Theorem \ref{HH}. For the sufficiency, suppose that $\M$ is faithful.  Let $\varphi:\pi_\A(Z(\mathcal{G}))\longrightarrow \pi_\B(Z(\mathcal{G}))$ be the isomorphism satisfying   $a\oplus\varphi(a)\in Z(\G)$ for all $a\in\A,$  whose existence guaranteed by the faithfulness of $\M.$ By virtue of $\rm (A')$, we   define $\ell_\A:\A  \longrightarrow Z(\A)$ and
$\ell_\B:\B \longrightarrow Z(\B)$  by $\ell_\A=\varphi^{-1}\circ h_\A$ and
$\ell_\B= \varphi\circ h_\B$. It is  obvious that  $\ell_\A(a)\oplus h_\A(a)\in Z(\G)$ and  $h_\B(b)\oplus \ell_\B(b)\in Z(\G),$ for all $a\in\A, b\in\B.$
 Further, $\rm (B')$ follows  that  \[\ell_\A(mn)=\varphi^{-1}(h_\A(mn))=h_\B(nm) \ {\rm  and}\ \ell_\B(nm)=\varphi(h_\B(nm))=h_\A(mn).\]
Now Theorem \ref{HH} together with Remark \ref {faith} (iii) confirm that $\K$ is proper, as required.
\end{proof}
\section{\color{SEC}Some sufficient conditions}
 We commence with the following result which  employs  Corollary \ref{H} to give a proof for a modification of  the main result of  Du and Wang \cite{DW}.  See also \cite[Corollary 1]{WW} and examine \cite[Theorem 2.1]{W} for $n=2$.
\begin{theorem}[{\cite[Theorem 1]{DW}}]\label{ideal}
Let $\mathcal{G}$ be a generalized matrix algebra with  faithful $\M$.
Then  $\mathcal{G}$  has Lie derivation property if
\begin{enumerate}[\hspace{1em}\rm (i)]
  \item $\pi_{\A}(Z(\mathcal{G}))=Z(\A), \pi_{\B}(Z(\mathcal{G}))=Z(\B)$, and
  \item either $\A$ or $\B$ does not contain nonzero central ideals.
\end{enumerate}
\end{theorem}
\begin{proof}
 Let $\K$ be a Lie derivation of the form presented in \eqref{gli}. From Corollary \ref{H}, as (i) implies $\rm (A')$, we only need  to show that $h_\B(nm)\oplus h_\A(mn)\in Z(\mathcal{G})$ for all $m\in\M, n\in\N.$ Without loss of generality suppose that $\A$ has no nonzero central ideals. Set
\[\gamma(a,b)=\ell_\A(a)+h_\B(b)\quad (a\in\A, b\in\B),\]
where, as in the proof of \ref{H},  $\ell_\A=\varphi^{-1}\circ h_\A.$ Then $p_\A=P-\ell_\A$ is a derivation. Now Proposition \ref{F} (e) implies that
\[p_\A(mn)=mg(n)+f(m)n-\gamma(mn,-nm)\quad (m\in\M, n\in\N),\]
so for each $a\in\A,\ \  p_\A(amn)-amg(n)-f(am)n=-\gamma(amn,-nam).$
By the latter identity and the fact that $p_\A$ is a derivation we  get
$$p_\A(a)mn+ap_\A(mn)-amg(n)-p_\A(a)mn-af(m)n=-\gamma(amn,-nam).$$
These relations follow that
$a\gamma(mn,-nm)= \gamma(amn,-nam),$
for  $a\in\A, m\in\M, n\in \N$. Therefore for any two elements $m\in\M, n\in\N$  the set $\A\gamma(mn,-nm)$
is a central ideal of $\A$. Thus $\ell_\A(mn)-h_\B(nm)=\gamma(mn,-nm)=0$ and so
$h_\B(nm)\oplus h_\A(mn)=\ell_\A(mn)\oplus h_\A(mn) \in Z(\mathcal{G})$, as claimed.
\end{proof}
 For an example of an algebra satisfying condition (ii) of Theorem \ref{ideal}, one can directly show  that every noncommutative unital prime algebra with a nontrivial idempotent  does not contain nonzero central ideals. In particular,  B(X), the algebra of operators on a Banach space $X$ when the dimension $X$ is grater that 1,  does not contain central ideal. The full matrix matrix algebra $M_n(A)$, $n\geq 2$, has also no nonzero central ideal, \cite[Lemma 1]{DW}. In \cite[Theorem  2]{DW} they also showed that in a generalized matrix algebra $\G$ with loyal  $\M$ if $\A$ is noncommutative then $\A$ has no central ideals.

In the next result, we provide  some new sufficient conditions assuring the Lie derivation property for $\G$. We say that  an algebra $\A$ is a domain if it has no zero devisors if $aa'=0$ implies $a=0$ or $a'=0$ for any two elements $a,a'\in\A.$
\begin{theorem}\label{domain}
Let $\mathcal{G}$ be a generalized matrix algebra with faithful  $\M$.
Then  $\mathcal{G}$  has Lie derivation property if
\begin{enumerate}[\hspace{1em}\rm (i)]
  \item $\pi_{\A}(Z(\mathcal{G}))=Z(\A),  \pi_{\B}(Z(\mathcal{G}))=Z(\B)$, and
  \item $\A$ and $\B$ are domain.
\end{enumerate}
\end{theorem}
\begin{proof} Let $\K$ be a Lie derivation of the form presented in \eqref{gli}.
From Corollary \ref{H}, we only need  to show that $h_\B(nm)\oplus h_\A(mn)\in Z(\mathcal{G})$.

Let $m\in\M, n\in\N.$ Using the identities in (c) of Proposition \ref{F} for $a=mn$ and $b=nm$ we get
\begin{equation}\label{f1}
P(mn)m-mh_\A(mn)+mnf(m)=mQ(nm)-h_\B(nm)m+f(m)nm.
\end{equation}
Multiplying  by $m$ the identities in  (e)  of Proposition \ref{F} we get
\begin{equation}\label{f2}
P(mn)m-h_\B(nm)m=mg(n)m+f(m)nm,\ {\rm and} \quad mQ(nm)-mh_\A(mn)=mg(n)m+mnf(m).
\end{equation}
Combining  the equations in  \eqref{f1} and  \eqref{f2} we get
 $2(h_\B(nm)m-mh_\A(mn))=0,$ so the $2-$torsion freeness of $\M$ implies that
 \begin{equation}\label{mn1}
 h_\B(nm)m=mh_\A(mn).
 \end{equation}
 The faithfulness of  $\M$  guaranties  the existence of an isomorphism  $\varphi:\pi_\A(Z(\mathcal{G}))\longrightarrow \pi_\B(Z(\mathcal{G}))$ satisfying   $a\oplus\varphi(a)\in Z(\G)$ for all $a\in\A.$ As  $\pi_{\B}(Z(\mathcal{G}))=Z(\B),$ we   can define $\ell_\A:\A  \longrightarrow Z(\A)$ and    $\ell_\B:\B  \longrightarrow Z(\B)$ by $\ell_\A=\varphi^{-1}\circ h_\A$ and $\ell_\B=\varphi\circ h_\B,$ respectively. It follows  that  $\ell_\A(a)\oplus h_\A(a)\in Z(\G)$  and  $h_\B(b)\oplus \ell_B(b)\in Z(\G)$ for all $a\in\A, b\in\B.$ It follows from   \eqref{mn1} that
  \begin{equation}\label{mn}
 (h_\B(nm)-\ell_\A(mn))m=0;
  \end{equation}
   or  equivalently,
 \begin{equation}\label{mnn}
 m(\ell_B(nm)-h_A(mn))=0.
 \end{equation}
 If $\N m=0$ and $m\N=0,$ then trivially $ h_\B(nm)-\ell_\A(mn)=0.$ Otherwise either   $m\N\neq0$  or  $\N m\neq0.$ Take  $0\neq n_0\in\N.$ If $mn_0\neq0$ then, as $\A$ is a domain, \eqref{mn} implies that $h_\B(nm)-\ell_\A(mn)=0;$ so
  \[h_\B(nm)\oplus h_\A(mn)=\ell_\A(mn)\oplus h_\A(mn) \in Z(\mathcal{G}).\]
  If $n_0m\neq0$  then, as $B$ is a domain,  \eqref{mnn}  implies that $\ell_B(nm)-h_A(mn)=0;$ so
 \[h_\B(nm)\oplus h_\A(mn)=h_\B(nm)\oplus \ell_B(nm) \in Z(\mathcal{G}).\]
 We therefore have
$h_\B(nm)\oplus h_\A(mn) \in Z(\mathcal{G})$ for all $m\in\M, n\in\N,$ as required.
\end{proof}
Here we impose a condition on $\M$ which is  slightly stronger than the faithfulness of $\M$. We say  an $(\A,\B)-$module $\M$ is strongly faithful if either
\begin{enumerate}[\hspace{1em}]
\item  $\M$ is faithful as a right $\B-$module and   $am=0$ implies  $a=0$ or $m=0$ for any $a\in\A, m\in\M$; or
\item $\M$ is faithful as a left $\A-$module and $mb=0$ implies  $m=0$ or $b=0$ for any $m\in\M, b\in\B.$
\end{enumerate}
It is evident that if $\M$ is strongly faithful then $\M$  is faithful and  either  $\A$ or $\B$ has no zero devisors. Following the proof of the above theorem we have the following result.
\begin{theorem}\label{strong}
A generalized matrix algebra $\G$ has  Lie derivation property if
\begin{enumerate}[\hspace{1em}\rm (i)]
  \item $\pi_{\A}(Z(\mathcal{G}))=Z(\A),  \pi_{\B}(Z(\mathcal{G}))=Z(\B)$, and
  \item  $\M$ is strongly faithful.
  \end{enumerate}
\end{theorem}
 \begin{proof} Following the proof of Theorem \ref{domain}, the equations \eqref{mn} and \eqref{mnn} together with the strong faithfulness of $\M$  ensure  that $h_\B(nm)-\ell_\A(mn)=0,$ so
  $h_\B(nm)\oplus h_\A(mn)=\ell_\A(mn)\oplus h_\A(mn) \in Z(\mathcal{G}).$ The conclusion now follows from Corollary \ref{H}.
 \end{proof}
We remark that, to the best of our knowledge, we do not know when one can drop  the strong faithfulness in Theorem \ref{strong}.\\

Similar to what was introduced in \cite[Section 3]{CH1}, we  also introduce a critical subalgebra $\mathcal{W}_\A$ of an algebra $\A.$
 With the same notations  as in Theorem \ref{HH}, suppose that $P$ is a Lie derivation on $\A$ and $\ell_\A:\mathcal{A}\rightarrow Z(\mathcal{A})$ is a linear map  such that $P-\ell_\A $ is a derivation on $\mathcal{A}.$ Set
 \[\V_{\mathcal{A}}=\{a\in\A: \ell_\A(a)\oplus h_\A(a)\in Z(\G)\}.\]
 In other words,  $\mathcal{V}_{\A}$ consist of those elements $a\in\A$ such that
 \[\ell_\A(a)m=mh_\A(a)\ \mbox{and}\  n\ell_\A(a)=h_\A(a)n,\ \mbox{for all} \ m\in\M, n\in\N.\]
 It is also easy to verify that $\mathcal{V}_{\A}\subseteq h_\A^{-1}(\pi_\B(Z(\G))),$
with  equality holds in the case where $\M$ is faithful. With some modifications in the proof of {\cite[Proposition 10]{CH1}}, one can show  that $\mathcal{V}_{\A}$ is a  subalgebra of $\A$ containing all commutators and idempotents. More  properties of $\mathcal{V}_{\A}$  are investigated in \cite{Mk}.

We denote by $\mathcal{W}_\A$ the smallest subalgebra of $\A$ contains all commutators and idempotents. We are especially dealing with those algebras satisfying $\mathcal{W}_\A=\A.$ If $\mathcal{W}_\A=\A$ then trivially $h_\A(\A)\subseteq \pi_\B(Z(\G))$, or equivalently,
$\pi_\B(\mathcal{L}(\A))\subseteq \pi_\B(Z(\mathfrak\G)).$ Some examples of algebras satisfying $\mathcal{W}_\A=\A$ are:
 the full matrix algebra $\A=M_n(A), n\geq 2, $ where $A$ is a unital algebra, and also every simple unital algebra $\A$ with a nontrivial idempotent.

Regarding the latter  observations and some suitable combinations of  the various  assertions in the results  \ref{ideal},  \ref{domain} and \ref{strong}, we apply Theorem \ref{HH} and Corollary \ref{H}   to arrive  the main result of the paper providing several sufficient conditions ensuring the Lie derivation property for a generalized matrix algebra $\G,$ which is a generalization of {\cite[Theorem 11]{CH1}}.
\begin{theorem}\label{main}
A generalized matrix algebra $\G$ has Lie derivation property if the following three conditions hold:
\begin{enumerate}[\hspace{1em}\rm (I)]
\item $\pi_\B(Z(\G))=Z(\B)$ and $\M$ is a faithful left $\A$-module; or $\A=\mathcal{W}_\A$ and $\M$ is a faithful left $\A$-module; or $\A$ has Lie derivation property and $\A=\mathcal{W}_\A.$
\item $\pi_\A(Z(\G))=Z(\A)$ and $\M$ is a faithful right $\B$-module; or $\B=\mathcal{W}_\B$ and $\M$ is a faithful right $\B$-module; or $B$ has Lie derivation property and $\B=\mathcal{W}_\B.$
\item One of the following assertions holds:
\begin{enumerate}[\hspace{1em}\rm (i)]
\item Either $\A$ or $\B$ does not contain nonzero central ideals.
\item $\A$ and $\B$ are domain.
\item Either $\M$ or $\N$ is strongly faithful.
\end{enumerate}
\end{enumerate}
\end{theorem}
\section{\color{SEC}Some Applications}
The main examples of generalized matrix algebras are: a triangular algebra ${\rm Tri}(\A,\M,\B)$, a unital algebra with a nontrivial idempotent,  the algebra $B(X)$ of operators on a Banach space $X$ and the full matrix algebra $M_n(\A)$ on a unital  algebra $\A$.
\subsection*{\color{SEC}Lie derivations on trivial generalized matrix algebras and  ${\rm Tri}(\A,\M,\B)$}
We say $\G$ is a trivial generalized matrix algebra when $\M\N=0$ and $\N\M=0.$ The main example of trivial generalized matrix algebra is the so-called triangular algebra ${\rm Tri}(\A,\M,\B)$ whose Lie derivation property has been extensively investigated by Cheung \cite{CH1, CH2}. As an immediate consequence of Corollary \ref{H} and Theorem \ref{main} we get the following corollary which generalizes {\cite[Corollary 7, Theorem 11]{CH1}}  to trivial generalized matrix algebras.
\begin{corollary}\label{trivial}
Let $\mathcal{G}$ be a trivial generalized matrix algebra and let
 $\K$ be a Lie derivation on $\mathcal{G}$  of the form presented in \eqref{gli}. If $\K$ is proper then
 $h_\A(\A) \subseteq \pi_\B(Z(\mathcal{G}))$ and $h_\B(\B)\subseteq \pi_\A(Z(\mathcal{G})).$
The converse is also hold in the case when $\M$ is faithful.

In particular, a trivial generalized matrix algebra $\G$ has Lie derivation property   if the following two conditions hold:
\begin{enumerate}[\hspace{1em}\rm (I)]
\item $\pi_\B(Z(\G))=Z(\B)$ and $\M$ is a faithful left $\A$-module; or $\A=\mathcal{W}_\A$ and $\M$ is a faithful left $\A$-module; or $\A$ has Lie derivation property and $\A=\mathcal{W}_\A.$
\item $\pi_\A(Z(\G))=Z(\A)$ and $\M$ is a faithful right $\B$-module; or $\B=\mathcal{W}_\B$ and $\M$ is a faithful right $\B$-module; or $\B$ has Lie derivation property and $\B=\mathcal{W}_\B.$
\end{enumerate}
\end{corollary}
The following example which has been adapted from {\cite[Example 3.8]{B}} presents a trivial generalized matrix algebra (which is not a triangular algebra) without the Lie derivation property.
\begin{example} Let $\M$ be a commutative unital algebra of dimension $3$ (on the commutative unital ring ${\bf R}$) with base $\{1, a_0, b_0\}$ such that ${a_0}^2={b_0}^2=a_0b_0=b_0a_0=0.$ Set $\N=\M$ and let $\A$ and $\B$ be the subalgebras of $\M$ generated by $\{1, a_0\}$ and $\{1, b_0\}$, respectively. Then $\M\N=0=\N\M$ so the generalized matrix algebra  $\G=\left(\begin{array}{cc}
\mathcal{A} & \M\\
\N & \B\\
\end{array}\right)$  is trivial, and a direct calculation reveals that the map $\K:\G\longrightarrow\G$ defined as \begin{eqnarray*}
\K\left(\begin{array}{cc}
r_1+r_2a_0 & s_1+s_2a_0+s_3b_0\\
t_1+t_2a_0+t_3b_0& u_1+u_2b_0\\
\end{array}\right)=\left(\begin{array}{cc}
 u_2a_0& -s_3a_0-s_2b_0\\
-t_3a_0-t_2b_0& r_2b_0\\
\end{array}\right),
\end{eqnarray*}
(where the coefficients are taken from the ring ${\bf R}$) is a non-proper Lie derivation. It is worthwhile mentioning that $Z(\G)=R\cdot 1_\G$ and $\pi_\A(Z(\G))=R\cdot 1_\A\neq Z(\A), \pi_\B(Z(\G))=R\cdot 1_\B\neq Z(\B)$. Furthermore, it is also easy to verify that  $\A\neq\mathcal{W}_\A, \B\neq\mathcal{W}_\B$. Therefore,  none of the conditions (I) and (II) of Corollary \ref{trivial} is hold. Nevertheless, $\A$ and $\B$ have Lie derivation property. It should  be also noted that $h_\A(\A)\nsubseteq \pi_\B(Z(\G)),\ h_\B(\B)\nsubseteq\pi_\A(Z(\G))$; as, $h_\A(a_0)=b_0$ and  $h_\B(b_0)=a_0.$

This example can be compared to that was  given by Cheung in \cite[Example 8]{CH1}. They clarify the same discipline, however, $\G$ is not a triangular algebra.
\end{example}
\subsection*{\color{SEC}Lie derivations on unital algebras with a nontrivial idempotent}
In the following we investigate the Lie derivation property for some unital algebras with a nontrivial idempotent.  Let  $\A$ be  a unital algebra with  a nontrivial idempotent $p$ and $q=1-p.$ Then  $\A$  enjoys the Peirce decomposition
{\small $\G=\left(\begin{array}{cc}
p\A q & p\A q\\
q\A p & q\A q\\
\end{array}\right),$}
as a generalized matrix algebra.  Applying Theorem \ref{main} for this generalized matrix algebra $\G$ we obtain the following result which partly improves the case $n=2$ of a result given  by Wang \cite[Theorem 2.1]{W} and Benkovi\v{c} \cite[Theorem 5.3]{B}.
\begin{corollary}\label{idempotent}
Let $\A$ be a $2-$torsion free unital algebra with  a nontrivial idempotent $p$ and $q=1-p.$ Then $\A$ has Lie derivation property if the following three conditions hold:
\begin{enumerate}[\hspace{1em}\rm (I)]
\item {\small $Z(q\A q)=Z(\A)q$ and $p\A q$ is a faithful left $p\A p-$module; or $p\A p=\mathcal{W}_{p\A p}$ and $p\A q$ is a faithful left $p\A p-$module; or $p\A p$ has Lie derivation property and $p\A p=\mathcal{W}_{p\A p}.$}
\item {\small $Z(p\A p)=Z(\A)p$ and $q\A p$ is a faithful right $q\A q-$module; or $q\A q=\mathcal{W}_{q\A q}$ and $q\A p$ is a faithful right $q\A q-$module; or $q\A q$ has Lie derivation property and $q\A q=\mathcal{W}_{q\A q}.$}
\item One of the following assertions holds:
\begin{enumerate}[\hspace{1em}\rm (i)]
\item Either $p\A p$ or $q\A q$ does not contain nonzero central ideals.
\item $p\A p$ and $q\A q$ are domain.
\item Either $p\A q$ or $q\A p$ is strongly faithful.
\end{enumerate}
\end{enumerate}
\end{corollary}
As a consequence of the above corollary we bring  the following result showing that the algebra $B(X)$ of bounded operators on a Banach space $X$ enjoys Lie derivation property. The same result has already proved by Lu and Jing \cite{LJ} for Lie derivable map at zero and idempotents by a completely different method. See also \cite{LL} for the properness of nonlinear derivations on $B(X).$
\begin{corollary}\label{K8}
Let $X$ be a Banach space of dimension greater than 2. Then $B(X)$ has Lie derivation property.
\end{corollary}
\begin{proof}
Set $\A=B(X)$. Consider a nonzero element $x_0 \in X$ and $f_0 \in X^*$ such that $f_0(x_0)=1,$  then $p=x_0\otimes f_0$
 defined by $y \longmapsto f_0(y)x_0$ is a nontrivial idempotent. A direct verification reveals that $\A$ satisfies the implications \eqref{faithful}. Indeed, if $pTp\A q=\{0\}$ for some $T\in\A$, then choose a nonzero element  $y\in q(X)$ such that $q(y)=y.$ Let $x\in X$ then there exists an operator   $S\in B(X)$ such that $S(y)=x$ (e.g. $S:=x\otimes g$ for some $g\in X^*$ with $g(y)=1$). We then get $pTp(x)=pTp(S(q(y)))=0.$
 Further, it can be readily verified that   $Z(\A)=\Bbb{C}I_X=\Bbb{C}(p+q)$, $Z(p\A p)=\Bbb{C}p$ and $Z(q\A q)=\Bbb{C}q.$ In particular, $Z(p\A p)=Z(\A)p$, $Z(q\A q)=Z(\A)q.$
 These also imply that neither $p\A p$ nor $q\A q$ has no central ideals. By Corollary \ref{idempotent}, $\A=B(X)$ has Lie derivation property.
  \end{proof}
We conclude this section with an application of Corollary \ref{idempotent} to the full matrix algebra $\A=M_n(A)$, $n\geq 2$, where $A$ is a $2-$torsion free unital algebra. Then $p=e_{11}$ is a nontrivial idempotent and  $q=e_{22}+\cdots e_{nn}.$ It is easy to verify that $p\A p=A, q\A q=M_{n-1}(A).$ As  $Z(\A)=Z(A)1_\A$ we get  $Z(p\A p)=Z(\A)p, Z(q\A q)=Z(\A)q,$ so both conditions (I) and (II) in Corollary \ref{idempotent} are fulfilled. Further, in  the case where  $n\geq 3,$ the algebra $q\A q=M_{n-1}(A)$ does not contain nonzero central ideals, so $\A=M_n(A)$ also satisfies assumption (i) of the aforementioned corollary.  Therefore  $M_n(A)$ has Lie derivation property for $n\geq 3;$ this is an adaptation of \cite[Corollary 1]{DW}. The same result was treated (especially for $n=2$)  in \cite[Corollary 5.7]{B}.

\end{document}